\documentclass[12pt,a4paper]{amsart}
\usepackage{enumerate}
\usepackage{pb-diagram}
\usepackage{microtype}
\usepackage{amsmath}
\usepackage{amssymb, pb-diagram}
\usepackage[all]{xy}
\usepackage{graphicx} 
\usepackage{mathabx}
\usepackage{chapterbib}
\usepackage{epsfig}
\usepackage{amsfonts}
\usepackage{amssymb}
\usepackage{amsmath}   
\usepackage{amsthm}
\usepackage{color}
\usepackage{latexsym}
\usepackage{hyperref}
\usepackage{amsrefs}
\usepackage{mathrsfs}

\input xy
\xyoption{all}
\usepackage{pgf, tikz}
\usetikzlibrary{arrows,shapes}

\renewcommand{\epsilon}{\varepsilon}

\newtheorem{theorem}{Theorem}[section]
\newtheorem{proposition}[theorem]{Proposition}
\newtheorem{corollary}[theorem]{Corollary}
\newtheorem{lemma}[theorem]{Lemma}
\newtheorem{conjecture}[theorem]{Conjecture}

\theoremstyle{definition}

\newtheorem{definition}[theorem]{Definition}

\newtheorem{question}[theorem]{Question}

\newtheorem*{question*}{Question}

\theoremstyle{remark}
\newtheorem{remark}[theorem]{Remark}

\DeclareMathOperator{\cd}{cd}
\DeclareMathOperator{\gd}{gd}

\newcommand{\FF}{\mathfrak F}
\newcommand{\GG}{\mathfrak G}

\newcommand{\Vc}{\mathfrak{Vc}}

\newcommand{\mS}{\mathcal S}
\newcommand{\M}{\mathrm{Mod}}

\newcommand{\Z}{\mathbb Z}
\newcommand{\N}{\mathbb N}
\newcommand{\R}{\mathbb R}

\newcommand{\finfty}{{\operatorname{F}}_{\infty}}

\newcommand{\UUF}{\underline{\underline{\operatorname{F}}}}
\newcommand{\UF}{\underline{\operatorname{F}}}

\newcommand{\uucd}{\mathop{\underline{\underline\cd}}}
\newcommand{\uugd}{\mathop{\underline{\underline\gd}}}
\newcommand{\ugd}{\mathop{\underline\gd}}

\newcommand{\Ho}{\operatorname{H}}

\newcommand{\cohom}[3]{H^{{\raise1pt\hbox{$\scriptstyle#1$}}}(#2\>\!,#3)}
\newcommand{\tatecohom}[3]%
  {\widehat H^{{\raise1pt\hbox{$\scriptstyle#1$}}}(#2\>\!,#3)}

\newcommand{\Cohom}[3]%
  {H^{{\raise1pt\hbox{$\scriptstyle#1$}}}\big(#2\>\!,#3\big)}
\newcommand{\Tatecohom}[3]%
  {\widehat H^{{\raise1pt\hbox{$\scriptstyle#1$}}}\big(#2\>\!,#3\big)}

\newcommand{\homol}[3]{H_{{\lower1pt\hbox{$\scriptstyle#1$}}}(#2\>\!,#3)}
\newcommand{\homolog}[2]{H_{{\lower1pt\hbox{$\scriptstyle#1$}}}(#2)}

\renewcommand{\ker}{\operatorname{Ker}}

\newcommand{\eg}{{\underline EG}}

\newcommand{\uueg}{{\underline{\underline E}}G}
\newcommand{\EFG}{E_{\frak F}G}
\newcommand{\EFiG}{E_{{\frak{F}}_i}G}
\newcommand{\EvcG}{E_{\Vc}G}
\newcommand{\ue}{{\underline E}}

\newcommand{\Hom}{\operatorname{Hom}}

\newcommand{\OFG}{\mathcal O_{\mathcal F}G}

\newcommand{\OVC}{\mathcal O_{\mathfrak{Vc}}}

\newcommand{\bs}{\smallsetminus}

\newcommand{\frakF}{\mathfrak{F}}
\newcommand{\frakG}{\mathfrak{G}}

\catcode`\@=11

\newcommand{\calO}{\mathcal O}
\newcommand{\ModOFG}{\mathop{{\operator@font
Mod\text{-}}\calO_{\frakF}G}}
\newcommand{\OFGMod}{\mathop{\calO_{\frakF}G\text{-}{\operator@font
Mod}}}
\newcommand{\ModOGG}{\mathop{{\operator@font
Mod\text{-}}\calO_{\frakG}G}}
\newcommand{\OGGMod}{\mathop{\calO_{\frakG}G\text{-}{\operator@font
Mod}}}

\newcommand{\Fall}{\frakF_{\operator@font all}}
\newcommand{\Ffin}{\frakF_{\operator@font fin}}
\newcommand{\Fvc}{\frakF_{\operator@font vc}}
\newcommand{\Fic}{\frakF_{\operator@font ic}}
\newcommand{\Ffg}{\frakF_{\operator@font fg}}
\newcommand{\Fpc}{\frakF_{\operator@font pc}}
\newcommand{\Fab}{\frakF_{\operator@font ab}}
\newcommand{\Fvpc}{\frakF_{\operator@font vpc}}
\newcommand{\Fvab}{\frakF_{\operator@font vab}}

\catcode`\@=12

\newcommand{\OC}[2]{\mathop{\mathcal{O}_{#2}#1}\nolimits}
\renewcommand{\OFG}{\OC{G}{\frakF}}

\DeclareMathOperator{\vcd}{vcd}

\newcommand{\ucd}{\mathop{\underline\cd}}
\renewcommand{\coprod}%
{\mathop{\rotatebox[origin=c]{180}{$\displaystyle\prod$}}\limits}

\numberwithin{equation}{section}

\begin{document}

\title[ classifying spaces for mapping class groups]{hierarchically cocompact classifying spaces for mapping class groups of surfaces}


\author{Brita Nucinkis}
\author{Nansen Petrosyan}
\address{Department of Mathematics, Royal holloway, University of London, Egham, TW20 0EX, UNITED KINGDOM}
\email{brita.nucinkis@rhul.ac.uk}
\address{School of Mathematics, University of Southampton, Southampton, SO17 1BJ, UNITED KINGDOM}
\email{N.Petrosyan@soton.ac.uk}

\curraddr{}

\curraddr{}

\thanks{The  second author was supported by the EPSRC First Grant EP/N033787/1.}



\date{\today}

\dedicatory{Dedicated to Peter Kropholler on the occasion of his 60th birthday.}


\begin{abstract} We define the notion of {\it a hierarchically cocompact} classifying space  for a family of subgroups of a group. Our main application is to show that the mapping class group $\M(S)$ of any  connected oriented compact surface $S,$ possibly  with punctures and boundary components and with negative Euler characteristic has a hierarchically  cocompact model for the family of virtually cyclic subgroups of  dimension at most $\vcd \M(S)+1$. When the surface is closed, we prove that this bound is optimal.  In particular, this answers a question of L\"{u}ck for mapping class groups of surfaces.
\end{abstract}

\maketitle

\section{Introduction}

Let $G$ be a group and denote by $\FF$ a family of subgroups of $G$, that is a collection of subgroups of $G$ closed under conjugation and finite intersection.
We denote by $\EFG$ the classifying space for the family $\FF$. A $G$-CW-complex $X$ is said to be a {\it model} for $\EFG$ if $X^H$ is contractible for all $H \in \FF,$ and $X^H$ is empty otherwise. The minimal dimension of a model for $\EFG,$ denoted by $\gd_{\FF}G,$ is called the {\it geometric dimension of $G$ for the family $\FF$}.

\begin{definition}\label{defwF} We say a $G$-CW-complex $X$ is a {\it hierarchically  cocompact (hierarchically finite type)} model for $\EFG$ if there are families $\FF_i$ ($i\in I$ some indexing set) such that
$\FF = \bigcup_{i\in I} \FF_i,$ and that there are cocompact (finite type) models for $\EFiG$ for all $i \in I.$
\end{definition}

By the universal property for classifying spaces for families, this is equivalent to saying that there is a model for $\EFG$ which is a mapping telescope of cocompact (finite type) models for $\EFiG.$ This follows from an argument analogous to \cite[Theorem 6.11]{martineznucinkis11}.

Definition \ref{defwF} was motivated by the following conjecture of Juan-Pineda and Leary.

\begin{conjecture}[Juan-Pineda--Leary, \cite{jpl}]\label{jpl-conj} Denote by $\Vc$ the family of virtually cyclic subgroups of a group G, and let $G$ be a group admitting a cocompact model for $\EvcG.$ Then $G$ is virtually cyclic.
\end{conjecture}

This conjecture has been proved in many cases, such as, for example, hyperbolic groups \cite{jpl}, elementary amenable groups \cite{gw}, one-relator groups, CAT(0)-groups, acylindrically hyperbolic groups, $3$-manifold groups \cite{vPW1}, and linear groups \cite{vPW2}. In all of these examples it was shown that these groups cannot admit a model for $\EvcG$ with a finite type $0$-skeleton, a condition equivalent to $G$ having a finite set of virtually cyclic subgroups $\{ H_1,...,H_n\}$ such that every virtually cyclic subgroup is sub-conjugate to $H_i$ for some $1\leq i\leq n$. T.~von Puttkamer and X.~Wu \cite{vPW2} actually conjecture that any finitely presented group satisfying this condition is already virtually cyclic. They also exhibit a finitely generated example of type $\underline{\underline{\operatorname{F}}}_0.$ Note, that non-finitely generated examples were already known, as for example \cite[6.4.6]{robinson} the famous construction of Higman-Neumann-Neumann with one conjugacy class of elements.

 In line with convention, we write $\eg$ for a classifying space for proper actions and  $\uueg$ for $\EvcG$. Denote the geometric dimension for proper action by $\ugd G$ and   $\gd_{\Vc}G$ by $\uugd G$. We also say that a group admitting a hierarchically cocompact model for $\uueg$ is of type $h\UUF.$ Finally, we use the notation $h\UUF_\infty$ for a group admitting a hierarchically finite type model for $\uueg.$

In this paper we provide a method by which to construct hierarchically cocompact (hierarchically finite-type) models for $\uueg$ out of 
cocompact (finite-type) models for $\eg,$ provided that  commensurators of virtually cyclic subgroups satisfy some further finiteness conditions. This enables us to show that many classes of finitely presented groups that do not satisfy the hypothesis of Conjecture \ref{jpl-conj}, in particular are not of type $\underline{\underline{\operatorname{F}}}_0$,  are still of type $h\UUF.$ Based on our observations of groups of type $h\UUF$, we ask the following question. 

\begin{question}\label{qu} Suppose a group $G$ is of type $h\UUF.$ Are commensurators of virtually cyclic subgroups of type $\operatorname{F}_\infty?$
\end{question}

\pagebreak

The main result of the paper is the following.

\begin{theorem}\label{main}
 Let $S$ be a compact connected orientable surface of genus $g$, with a finite number of boundary components  and punctures, and with negative Euler characteristic. Then the mapping class group $\M(S)$ has a hierarchically  cocompact model for ${\underline{\underline E}} \M(S)$ of dimension $\vcd \M(S)+1$.
Moreover, if $S$ is closed and  $g\geq 1$, then $$\uugd \M(S)=\uucd \M(S)=\vcd \M(S) +1.$$
\end{theorem}

Since  $\ugd \M(S)= \vcd \M(S)$ (see \cite{AM}), we note that this answers a question by Wolfgang L\"uck which asks  for which countable groups  $G$ the inequality 
$$\ugd G-1\leq \uugd G\leq \ugd G+1$$ holds, see for example \cite[Problem 9.51]{luecknewbook}. By Theorem \ref{main},  the above inequality is always satisfied for the mapping class group of any compact connected orientable surface $S$ with a finite number of boundary components  and punctures, and and $\chi(S)<0$.

Roughly stated, the hierarchically  cocompact model for ${\underline{\underline E}} \M(S)$ can be obtained by attaching  certain fibred spaces to the Teichm\"uller space of $S$. These are associated to the Weyl groups of the  infinite cyclic subgroups $H=\langle f \rangle \leq \M(S)$ representing the set $I$ of the complete sub-conjugacy classes of infinite virtually cyclic subgroups of $\M(S)$.  These spaces fibre over products of a Euclidean space and a Teichm\"uller space corresponding to the pseudo-Anosov and the trivial components in the canonical reduction system of $f\in \M(S)$ respectively. Since the attaching maps are equivariant,  Teichm\"uller spaces could be replaced by spines or the relevant minimal models for the local subgroups to obtain the desired hierarchically  cocompact model of minimal dimension. 

It is worth pointing out, that if one applies the above attaching construction to only a finite subset $J\subset I$, then the resulting space will be a cocompact model for the classifying space of $\M(S)$ for the family of all virtually cyclic subgroups that are sub-conjugate to a subgroup in $J$. 

Finite dimensional models for ${\underline{\underline E}} \M(S)$ have been exhibited by Degrijse and the second author in \cite{DP}  for closed surfaces $S$ (the obtained bound on  dimension is $9g-8$) and later by Juan-Pineda and Trufilo-Negrete in \cite{JT} for surfaces $S$ that have negative Euler characteristic with possible punctures and boundary components (the obtained bound is $[\M(S):\M(S)[m]](\vcd \M(S)+1)$, $m\geq 3$ where $\M(S)[m]$ is the level $m$ congruence subgroup). Apart from
exhibiting models that are mapping telescopes of cocompact models, our bounds substantially improve on the bounds given there. In particular, for closed surfaces, our bounds are optimal.

Furthermore, recently Bartels and Bestvina showed that mapping class groups of oriented surfaces $G$ of finite type satisfy the Farrell-Jones conjecture \cite{BB}. For example, the assembly map, which is induced by the $G$-map $\uueg \to \{pt\}$,
$$H_n^G(\uueg; \mathbf{K}_\Z) \xrightarrow{\cong} K_n(\Z G),$$
 is an isomorphism, where $H^G_n(-; \mathbf{K}_\Z)$ is the Bredon homology theory with the $K$-theory functor coefficients. In particular, since homology commutes with colimits and when $\chi(S)<0$, $G=\M(S)$ is of type $h\UUF$, the left-hand side of the isomorphism is a colimit of, hopefully, more computable homological terms. For a detailed introduction into the Farrell-Jones conjecture the reader is referred to \cite{luecknewbook}.

In the next section, we define what exactly is meant by a hierarchically cocompact model for a classifying space, and give a recipe for constructing these models. We also exhibit some examples in Section 3. Since we will need some results on Bredon cohomological dimensions later when giving lower bounds for the dimensions of our models, we give a brief introduction into Bredon cohomology in Section 4. Section 5 is devoted to proving our main result, Theorem \ref{main}.

\section{Constructing hierarchically cocompact models}

The following construction, due to L\"uck and Weiermann \cite{lueckweiermann}, will play a crucial role in what is to come:

Let $\FF$ and $\GG$ be families of subgroups of a given group $G$ such that $\FF \subseteq \GG$.

\begin{definition}\label{equiv}\cite[(2.1)]{lueckweiermann} A {\it strong equivalence relation} on $\GG\smallsetminus\FF$, denoted $\sim$, is an equivalence relation on $\GG\smallsetminus\FF$ satisfying:
\begin{itemize} 
\item For $H,K \in \GG\smallsetminus\FF$ with $H \leq K$ we have $H \sim K.$
\item Let $H,K \in \GG\smallsetminus\FF$ and $g \in G$, then $H \sim K \iff gHg^{-1} \sim gKg^{-1}.$
\end{itemize}
Denote by $[\GG\bs\FF] $ the equivalence classes of $\sim$ and define for all $[H]\in [\GG\bs\FF]$ the following subgroup of $G$:
$$N_G[H] =\{g \in G\,|\, [gHg^{-1}]=[H]\}.$$
Now define a family of subgroups of $N_G[H]$ by
$$\FF[H] =\{K \leq N_G[H] \,|\, K \in \GG\bs\FF \, ,\, [K]=[H]\}\cup (\FF \cap N_G[H]),$$
where $\FF \cap N_G[H]$ is the family of subgroups of  $N_G[H]$ that are in $\FF$.
\end{definition}

\begin{remark}\label{comm} A typical example of a pair of families that has a strong equivalence relation is the finite and the virtually cyclic subgroups of $G$. Here, the equivalence relation is the commensurability. 
\end{remark}

We need the following theorem of  L\"{u}ck and Weiermann.

\begin{proposition}\label{lw-main}\cite[Theorem 2.3]{lueckweiermann} Let $\FF \subseteq \GG$ be families with a strong equivalence relation on $\GG\smallsetminus\FF$. Denote by $I$ a complete set of representatives of the conjugacy classes in $[\GG\bs\FF].$ Then the $G$-CW-complex given by the cellular $G$ push-out
$$\xymatrix{\bigsqcup_{[H]\in I} G \times_{N_G[H]}E_{\FF \cap N_G[H]}N_G[H]  \ar[r]^{\phantom{ppppppppppppp}i} \ar[d]_{\bigsqcup_{[H]\in I} id_G \times_{N_G[H}f_{[H]}}  & E_{\FF}G \ar[d] \\
\bigsqcup_{[H]\in I} G \times_{N_G[H]}E_{\FF[H]}N_G[H] \ar[r]  &X}$$
where either $i$ or the $f_{[H]}$ are inclusions, is a model for $E_\GG G.$
\end{proposition}

The condition on the two maps being inclusions is not a big restriction as one can replace the spaces by the mapping cylinders, see \cite[Remark 2.5]{lueckweiermann}.

We also need the following transitivity principle.

\begin{proposition}\label{trans}\cite[Proposition 5.1]{lueckweiermann} Let $\FF \subseteq \GG$ be families of subgroups of a group $G$. Assume that $G$ admits a cocompact (finite type) model  $X_{\GG}$ for $E_\GG G$ and that every subgroup $H \in \GG$ admits a cocompact (finite type) model $X_{\FF\cap H}(H)$ for $E_{\FF\cap H}H.$ Then $G$ admits a cocompact (finite type) model for $E_\FF G$ of dimension $$\dim X_{\GG}+ \sup \{\dim X_{\FF\cap H}(H) \;| \; H\in \GG\}.$$
\end{proposition}

\begin{theorem}\label{comm-thm} Let $\FF \subseteq \GG$ be families with a strong equivalence relation on $\GG\smallsetminus\FF$. Suppose  that for every $H \in \GG\smallsetminus\FF$, $N_G[H]$ admits a cocompact (finite type) model $X_{\FF\cap H}(N_G[H])$ for $E_{\FF\cap H} N_G[H].$ Further assume that $G$ admits a cocompact (finite type) model $X_{\FF}$ for $E_\FF G$ and that for each $H \in \GG$, there is a cocompact (finite type) model $X_{\FF[H]}(N_G[H])$ for $E_{\FF[H]}N_G[H]$, then $G$ admits a hierarchically cocompact (hierarchically finite type) model for $E_\GG G$ of dimension
$$\sup \{\dim X_{\FF}, \dim X_{\FF\cap H}(N_G[H])+1, \dim X_{\FF[H]}(N_G[H]) \;| \; H\in \GG\smallsetminus\FF\}.$$
\end{theorem}

\begin{proof} Let $\mathcal S$ be a set of representatives of some finite number of conjugacy classes in $[\GG\bs\FF]$ and denote $$\widehat\FF[\mathcal S] :=\{K \leq G \,|\, K \in \GG\bs\FF \, ,\, \exists g \in G, H\in \mathcal S, [K^g]=[H]\}\cup \FF.$$ 

Consider the $G$-push-out of Proposition \ref{lw-main} replacing $\GG$ by $\widehat\FF[\mathcal S]$:

$$\xymatrix{\bigsqcup_{[H]\in \mS} G \times_{N_G[H]}X_{\FF\cap H}(N_G[H])  \ar[r]^{\phantom{ppppppppppppp}i} \ar[d]_{\bigsqcup_{[H]\in \mS} id_G \times_{N_G[H]}\iota_{[H]}}  & X_{\FF}\ar[d] \\
\bigsqcup_{[H]\in \mS} G \times_{N_G[H]}\mbox {cyl}(f_{[H]})  \ar[r]  &X}$$
 
 where $\iota_{[H]}$ is the natural inclusion, and $f_{[H]}: X_{ \FF\cap H}(N_G[H])  \to X_{ \FF[H]}(N_G[H])$ is given by the universal property for classifying spaces for a family. $X$ is now a cocompact (finite type) model for $E_{\widehat\FF[\mS]} G,$  as it is a $G$-push-out of cocompact (finite type) complexes. The result now follows from the fact that $\GG=\bigcup_{\mS\subseteq I, |\mS|<\infty}\widehat\FF[\mS]$ and the remark after Definition \ref{defwF}.
\end{proof}
We denote $\widehat\FF[H] :=\{K \leq G \,|\, K \in \GG\bs\FF \, ,\, \exists g \in G,  [K^g]=[H]\}\cup \FF$ and  note that $\widehat\FF[H]=\widehat\FF[\mS]$ exactly when $\mS=\{H\}$. We end this section with a useful lemma.

\begin{lemma}\label{fhat-lemma} Suppose $\FF \subseteq \GG$ are families of subgroups of a group $G$, and that for each $H \in \GG\bs\FF,$ $G$  admits  cocompact (finite type) models for both $E_{\widehat\FF[H]}G$ and $E_{\FF\cap H}H$. Then $G$ admits a hierarchically cocompact (finite type) model for $E_\GG G.$
\end{lemma}

\begin{proof}  For any two families $\FF_1$ and $\FF_2$ of subgroups of $G$,  we have  unique maps (up to $G$-homotopy) $E_{\FF_1 \cap \FF_2}G \to E_{\FF_1}G$ and $E_{\FF_1 \cap \FF_2}G \to E_{\FF_2}G.$ The double mapping cylinder $X$ gives a model for $E_{\FF_1 \cup \FF_2}G,$  and that if $E_{\FF_1 \cap \FF_2} G$, $E_{\FF_1} G$ and 
$E_{\FF_2} G$ are cocompact (finite type), then so is $X.$

Now, let $[H_1] \neq [H_2] \in I$ and $\FF_1 = \widehat\FF[H_1],$ and $\FF_2 = \widehat\FF[H_2].$  Then, obviously $\FF_1 \cap \FF_2 = \FF$ and $\FF_1\cup \FF_2 = (\widehat\FF[H_1])[H_2] =: \widehat\FF[H_1,H_2].$ 
By assumption there are cocompact (finite type) models for $E_{\FF_1}G$ and $E_{\FF_2}G,$ and the transitivity result in Proposition \ref{trans} yields that there is also a cocompact (finite type) model for  $E_{\FF_1 \cap \FF_2}G$. The double mapping cylinder $X$ gives a cocompact (finite type) model for $E_{\FF_1 \cup \FF_2}G$.

For each $i>2$, let  ${\widehat\FF[H_1,...,H_i]}:= {\widehat\FF[H_1,...,H_{i-1}]}[H_i]$. Proceed by induction, to construct a cocompact (finite type) model for $E_{\widehat\FF[H_1,...,H_i]}G$ from cocompact  (finite type) models for $E_{\widehat\FF[H_1,...,H_{i-1}]}G$ and $E_{\widehat\FF[H_i]}G$ as above. 
\end{proof}

\section{Some examples}

In the next examples, let $\FF$ and $\Vc$ be the families of finite and virtually cyclic subgroups of a given group $G$, respectively, equipped with the commensurability relation on $\Vc \smallsetminus \FF$ (see Remark \ref{comm}).

\subsection{Thompson groups}  In \cite{martineznucinkis11} the authors introduced a similar condition for the family of finite subgroups, quasi-$\UF_\infty,$ which asks for a group to have, for any $k \in \Z_{>0},$  finitely many conjugacy classes of finite subgroups of order $k$, and that normalisers of all finite subgroups are of type $\finfty.$  \cite[Theorem 4.9]{mmn14} shows that generalised Thompson groups, which are automorphism groups of valid, bounded and complete Cantor-algebras are quasi-$\UF_\infty$ and hence are $h\UF_\infty.$

\subsection{Hyperbolic groups} Following \cite{lueckweiermann}, a group is said to satisfy $(NM_{\Ffin \subseteq \Vc}),$ if every virtually cyclic subgroup is contained in a unique maximal virtually cyclic group $V$, such that $N_G(V) =V.$ Then, by \cite[Corollary 2.11]{lueckweiermann}, the push-out in Proposition \ref{lw-main} reduces to 
$$\xymatrix{ G \times_{V}\underline{E}V  \ar[r] \ar[d] & \eg \ar[d] \\
G/V \ar[r]  &X,}$$
where $X$ is a cocompact model for $E_{\widehat{\Ffin}[V]}G$, provided there is a cocompact model for $\underline{E}G$. Now, apply
Lemma \ref{fhat-lemma}, to conclude that any group satisfying $(NM_{\Ffin \subseteq \Vc})$ that has a cocompact model for $\underline{E}G$  is of type $h\UUF.$
In particular, hyperbolic groups satisfy $(NM_{\Ffin \subseteq \Vc})$ \cite[Remark 7]{jpl}, and hence are of type $h\UUF.$

\subsection{Polycyclic groups} Polycyclic groups are of type $h\UUF.$ This follows from \cite[Lemma 5.15]{lueckweiermann} - for every infinite virtually cyclic group $V$, there exists a cyclic subgroup $C$ commensurate to $V$ such that $N_G[C]=N_G(C),$ and hence a model for $E_{\FF[V]}N_G[V]$ is given by $\underline{E}(N_G(C)/C)$ by pulling back the action. Now, since $G$ is polycyclic, we have a cocompact $\eg$ and a cocompact $\underline{E}(N_G(C)/C)$ for every $C$. Now apply Theorem \ref{lw-main}.

\subsection{Soluble Baumslag-Solitar groups} One can also show that the soluble Baumslag-Solitar groups $G=BS(1,n)$ are of type $h\UUF.$ The group $G$ is torsion-free and admits a cocompact $2-$dimensional model for $\eg=EG.$ Next, we show that for all cyclic subgroups $H \leq G$, we have a finite model $X_{\FF[H]}(N_G[H])$ for $E_{\FF[H]}N_G[H]$ and apply Theorem \ref{comm-thm}. 
Recall,
$$BS(1,n) \cong \Z[\tfrac1n] \rtimes \Z.$$

Any infinite cyclic subgroup $H$ not contained in $\Z[\tfrac1n],$ is equal to its commensurator, see \cite[Lemma 5]{fluch11}, and hence
$N_G[H]$ has a point as a model for $E_{\FF[H]}N_G[H],$ and $\R$ is a cocompact model for $\eg.$ Infinite cyclic subgroups of $ \Z[\tfrac1n]$, on the other hand, are all commensurate to $K=\langle \tfrac 1n\rangle \leq  \Z[\tfrac1n]$. $K$ has the entire  $G$ as its commensurator. Since $G$ is an HNN-extension with the vertex group $K$, the associated Bass-Serre tree is a cocompact model for $E_{\FF[K]}G.$ Applying Theorem \ref{comm-thm}, we obtain that $BS(1,n)$ is of type  $h\UUF$ for all $n>0$.

Note that when $n>1$, the normaliser $N_G(\langle \tfrac 1n \rangle) = \Z[\tfrac1n]$ is not even finitely generated. However, in all examples in this note, the commensurators of virtually cyclic subgroups are of type $\operatorname{F}_\infty.$ It is not clear whether this is just an artefact of our construction, and hence we ask, see Question \ref{qu}:

\begin{question*} Suppose a group $G$ is of type $h\UUF.$ Are commensurators of virtually cyclic subgroups of type $\operatorname{F}_\infty?$
\end{question*}

\section{Bredon cohomology}

In this section we introduce all necessary facts and results regarding Bredon cohomology needed later on to determine lower bounds for dimensions of our classifying spaces. All results in this section are well known; a good introduction to the subject can be found in \cite{fluch}.

\noindent
Let $\FF$ denote a family of  subgroups of a given group $G$ as before.
We consider the category $\OFG,$ which has as objects the transitive $G$-sets with stabilisers in $\FF$. Morphisms in $\OFG$ are $G$-maps between those $G$-sets.  Modules over the orbit category, called $\OFG$-modules are contravariant functors from the orbit category to the category of abelian groups.  Exactness is defined pointwise: a sequence
 $$A\to B\to C$$
  of $\OFG$-modules
is exact at $B$ if and only if the sequence of abelian groups
$$A(G/K)\to B(G/K)\to C(G/K)$$
is exact at $B(G/K)$ for every $G/K \in \OFG.$
The trivial $\OFG$-module, which is denoted $\Z(-)$, is the constant functor $\Z$ from $\OFG$ to the category of abelian groups.

\noindent
The category $\OFG$-Mod of $\OFG$-modules has  enough projectives. Hence we can consider a projective resolution
$$ P_*(-) \to \Z(-) \to 0$$
of the trivial module $\Z(-).$ 
The Bredon cohomology functors $\Ho^*_\FF(G,-)$ are defined as derived functors of
$\Hom_\FF(\Z(-),-)$.  In particular,   for each $N=N(-) \in \OFG$-Mod,
$$ \Ho^*_\FF (G,N) =H_*(mor(P_*,N)).$$

We can now define cohomological dimensions in the Bredon-setting with analogous properties to ordinary cohomology. 
Let $n \geq 0$. We say $G$ has Bredon-cohomological dimension $\cd_\FF G \leq n$ if there is a projective resolution of $\Z(-)$ of length $n.$ This is equivalent to $ \Ho^{n+1}_\FF (G,N)=0$ for all $\OFG$-modules $N.$

Let $X$ be a model for $\EFG$. Then the cellular chain complex $C_*(X)$ gives rise to a free resolution 
$$C_*(X)(-) \to \Z(-)\to 0$$
by putting $C_*(X)(G/K)=C_*(X^K)$ for all $K \in \OFG.$ In particular,
$$\cd_\FF G \leq \gd_\FF G.$$
Furthermore, \cite[Theorem 13.19]{lueckbook}, if $\cd_\FF G \geq 3,$ then $\gd_\FF G =\cd_\FF G.$

As before, for $\FF$ the family of finite groups, we write $\ucd G$ and $\ugd G$, and for $\FF$ the family of virtually cyclic subgroups we write $\uucd G$ and $\uugd G$ respectively.  Note that for torsion-free groups, $\ucd G =\cd  G$ and $\ugd G =\gd G$ and Bredon-cohomology for the family of finite subgroups becomes ordinary cohomology over the group ring $\Z G.$

We will quite often make use of this simple observation. When $G$ is virtually torsion-free, then $\ucd G\geq \vcd G$. So, if $G$ admits a model for $\ue G$ of dimension $\vcd G$, then it must necessary be of minimal dimension, i.e.~realising  $\ugd G$.

\section{Mapping class groups}

In this section, we prove our main theorem and show that the mapping class group of any compact orientable surface $S$ with possibly finitely many punctures and boundary components,  and with negative Euler characteristic $\chi(S)$ is of type  $h\UUF$ with a hierarchically cocompact model of dimension $\vcd \M(S)+1$. We recall some necessary background on mapping class groups of surfaces and refer the reader to \cite{FM} and \cite{ivanov} for further details.

Let $S$ be a connected compact oriented surface with finitely many punctures and $\chi(S)<0$. The {\it mapping class group} of $S$, denoted by $\M(S)$, is the group of isotopy classes of orientation  preserving diffeomorphisms of $S$ pointwise fixing the boundary $\partial S$ 
$$\M(S) = \mbox{Diff}^+(S,\partial S)/{\mbox{Diff}^0(S,\partial S)},$$
where $\mbox{Diff}^0(S, \partial S)$ is the subgroup of $\mbox{Diff}^+(S, \partial S)$ consisting of elements that are isotopic to the identity. 

Any diffeomorphism of $S$ induces an automorphism of $H_1(S, \Z/{m\Z})$ for $m\geq 2$. This gives a well-defined homomorphism 
$$\M(S)\to \mbox{Aut}(H_1(S, \Z/{m\Z}))$$ where the kernel is denoted by $\M(S)[m]$ and it is called the {\it level $m$ congruence subgroup} of $\M(S)$. 

Suppose now that $S$ has no boundary. Let $\{\alpha_1, \dots, \alpha_n\}$ be a collection of pairwise disjoint, homotopically distinct essential simple closed curves in $S$. Denote by $\sigma= \{[\alpha_1], \dots, [\alpha_n]\}$ the corresponding isotopy classes.

{\it Inclusion homomorphism}.~({\cite[Thm.~3.18]{FM}})\label{inclusion}. Let $N_{\sigma}$ be an open regular neighbourhood of $\cup_{j=1}^n\alpha_j$ in $S$. Denote $S_{\sigma}=S\smallsetminus N_{\sigma}$ and set $S_{\sigma}=\cup_{i=1}^k S_i$ where each $S_i$ is a connected subsurface. Let $\{\beta_1, \gamma_1\}, \dots, \{\beta_k, \gamma_k\}$ denote the pairs of boundary components of $S_{\sigma}$ that bound the annuli in $N_{\sigma}$.  The inclusion $S_{\sigma} \hookrightarrow S$ induces a homomorphism 
$$\eta_{\sigma}:\M(S_{\sigma})=\prod_{i=1}^k\M(S_i) \to \M(S)$$ 
with kernel  $\langle T_{\beta_1}T^{-1}_{\gamma_1}, \dots, T_{\beta_k}T^{-1}_{\gamma_k}\rangle$. The restriction $\eta_i:=\eta_{\sigma}|_{\M(S_i)}$ is the map induced by the inclusion $S_i\hookrightarrow S$.

{\it Capping homomorphism}.~({\cite[Prop.~3.19]{FM}})\label{capping} Let $\widehat {S}_{\sigma}$ be the surface obtained from $S_{\sigma}$ by capping the boundary components with once-punctured disks and write $\widehat {S}_{\sigma}=\cup_{i=1}^k \widehat{S}_i$. The inclusion $S_{\sigma} \hookrightarrow \widehat {S}_{\sigma}$ induces a homomorphism 
$$\theta_{\sigma}:\M(S_{\sigma}) \twoheadrightarrow \prod_{i=1}^k \M(\widehat S_i, \Omega_i)\subseteq \M(\widehat{S}_{\sigma})$$
with kernel  $\langle T_{\beta_1}, T_{\gamma_1}, \dots, T_{\beta_k}, T_{\gamma_k} \rangle$.  Here $\Omega_i$ denotes the set of punctures coming from the boundary components of $S_i$. The image of the restriction $\theta_i:=\theta_{\sigma}|_{\M(S_i)}$, denoted  $\M(\widehat S_i, \Omega_i)$, is the subgroup of $\M(\widehat{S}_i)$ consisting of all the elements that fix $\Omega_i$ pointwise.  $\M(\widehat S_i, \Omega_i)$ contains the pure mapping class group and hence is finite index in $\M(\widehat{S}_i).$  

{\it Cutting homomorphism}.~({\cite[Prop.~3.20]{FM}})\label{cutting} Define $\M(S)_{\sigma}=\{ g\in \M(S) \;|\; g(\sigma)=\sigma\}$. There is a well-defined homomorphism
$$\rho_{\sigma} : \M(S)_{\sigma}\to \M(S \smallsetminus \cup_{j=1}^n\alpha_j)= \M(\widehat{S}_{\sigma})$$
with free abelian kernel  $\langle T_{\alpha_1}, \dots, T_{\alpha_n}\rangle$ generated by the Dehn twists about the curves $\alpha_1, \dots, \alpha_n$. Let $\M(S)_{\sigma}^0$ be the finite index subgroup of $\M(S)_{\sigma}$ consisting of all the elements that fix each curve $\alpha_i$ with orientation. Denote the restriction $\rho_{\sigma, 0}=\rho_{\sigma}|_{\M(S)_{\sigma}^0}$. Then $\theta_{\sigma}=\rho_{\sigma, 0}\circ \eta_{\sigma}$ (see the diagram on page 91 of \cite{FM}).

{\it The canonical form}.~({\cite[Cor.~13.3]{FM}})\label{canonical} Let $f\in \M(S)$ and let $\sigma=\sigma(f)$ be its canonical reduction system (see \cite[\S13.2.2]{FM}). Let $S_{k+1},\dots, S_{k+n}$ be the pairwise disjoint annuli that are the closed neighbourhoods of the curves $\alpha_1, \dots, \alpha_n$ representing $\sigma$. Then there is a representative $\phi$ of $f$ that permutes the $S_l$, so that some power of $\phi$ leaves invariant each $S_l$, $1\leq l \leq n+k$. By applying Nielsen-Thurston Classification Theorem to each $S_l$, one obtains that  there exists $p > 0$  so that $\phi^p(S_l) = S_l$ for all $1\leq l \leq n+k$ and
\begin{equation}\label{form} f^p = \prod_{i=1}^{k}\eta_i(f_i)\prod_{j=1}^{n}T_{\alpha_j}^{n_j}\end{equation}
where each  $f_i\in \M(S_i)$  is either pseudo-Anosov or the identity and $n_j\in \N$ for $1\leq i\leq k$, $1 \leq  j \leq n$.

\begin{remark}\label{ivanov} By a result of Ivanov \cite[Corollary 1.8]{ivanov}, if $f\in \M(S)[m]$, $m\geq 3$, then the integer $p$ in (\ref{form}) can always be taken to be one. \end{remark}

\begin{remark}[{\cite[3.10]{JT}}]\label{ivanov2} By Theorem 1.2 in \cite{ivanov}, $\M(S)[m]_{\sigma}\subseteq \M(S)_{\sigma}^0$. Since $\theta_{\sigma}=\rho_{\sigma, 0}\circ \eta_{\sigma}$, it follows that $\theta_i(f_i)=\hat{f}_i$ for $1\leq i \leq k$.
\end{remark}

For simplicity, we will often denote $G=\M(S)$.  Let $\FF$ and $\Vc$ the families of finite and virtually cyclic subgroups of $G$, respectively, equipped with the commensurabilty relation on $\Vc \smallsetminus \FF$ (see Remark \ref{comm}). The Teichm\"uller space $\mathcal T(S)$ is a CAT(0)-space on which $G$ acts properly and cocompactly.  So, it gives a cocompact model for $\underline{E}G$.  \cite[Corollary 1.3]{AM} of Aramayona and Mart\'\i nez-P\'erez states that there is a cocompact model for $\underline{E}G$ of minimal dimension $\ugd G =\vcd G$. We need the following minor generalisation of this result.

\begin{proposition}\label{non-conn} Let $S$ be a compact (possibly disconnected) surface with possible punctures and boundary components. Then there is a cocompact model for $\ue \M(S)$ of dimension $\ugd \M(S) =\vcd \M(S)$.
\end{proposition}
\begin{proof} In \cite{AM}, this has been proven for connected $S$. Suppose $S$ decomposes into the disjoint union of diffeomorphic copies of its connected components 
$$S=\sqcup^{m_1}  S_{1}\dots \sqcup^{m_q} S_{q}.$$ 
    Subsequently, this implies that  $\M(S)\subseteq \prod_{i=1}^q\M(S_{i})\wr \Sigma_{m_i}$ and is of finite index. Combining  \cite[Theorem 4.1]{harer} and the extension theorem for duality groups in \cite[Thm.~3.5]{BE}, gives us that $$\vcd \M(S)=\sum_{i=1}^q m_i\cdot \vcd \M(S_i).$$ The next lemma shows that there exists a cocompact model for $\ue \M(S)$ of dimension $\vcd \M(S)$. Since $\vcd \M(S)\leq \ugd \M(S)$, the result follows.
\end{proof}

\begin{lemma}\label{fin_wr}  Let  $K$ be a group.  Suppose $K$ has a cocompact model $X$ for $\ue K$. Then the wreath product $W=K\wr \Sigma_m$ has a cocompact model for $\ue W$ of dimension $m\cdot \dim X$.
\end{lemma}
\begin{proof} The wreath product $W$ acts on  $Y=X\times \dots \times X$ which consists of $m$-copies of $X$ with a diagonal action of $K^m$ and  a permutation action of $\Sigma_m$. Note that $W$ acts cocompactly on $Y$ and it is a model for $\ue W$ of dimension $m\cdot \dim X$.
\end{proof}

To establish Theorem \ref{main},  it remains to show, see Theorem \ref{comm-thm}, that for each infinite cyclic subgroup $H \subseteq G$,  there are cocompact models for $E_{\FF\cap N_G[H]}N_G[H]$ and $E_{\FF[H]}N_G[H]$ where $N_G[H]$ is the commensurator of $H$ in $G$ of the required dimensions.

\begin{proposition}[{\cite[Prop.~4.8]{JT}}]\label{pin} Let $S$ be an orientable closed surface with finitely many punctures and $\chi(S)< 0$. Suppose $f\in G$ generates an infinite cyclic subgroup $H$. Then for any integer $l>0$ such that $f^l\in \M(S)[m]$, $m\geq 3$, $N_G[H]=N_G(f^l)$ holds.
\end{proposition}

We need the following slight generalisation of   \cite[Proposition 4.12]{JT} of Juan-Pineda and Trufilo-Negrete.

\begin{proposition}\label{sequence} Suppose $S$  is an orientable closed surface with finitely many punctures and $\chi(S)< 0$. Let $f\in \M(S)[m]$, $m\geq 3$,  with the canonical reduction system $\sigma$ with $\hat f_{a+1},  \dots , \hat f_{k}$ pseudo-Anosov and
$$\rho_{\sigma}(f):=\left(id_1, \dots, id_a, \hat f_{a+1},  \dots , \hat f_{k}\right)\in \prod_{i=1}^k \M(\widehat S_i, \Omega_i),$$ 
then there is a central extension
\begin{equation}\label{restrict_cent} 1\to \Z^n\to C_{\M(S)}(f)^0{\xrightarrow{\rho_{\sigma}}}\prod_{i=1}^a \M(\widehat S_i, \Omega_i)\prod_{j=a+1}^k V_j \to 1.\end{equation}
where $V_j=C_{\M(\widehat S_j, \Omega_j)}(\hat f_{j})$ is virtually cyclic for each $a+1\leq j\leq k$.
Also,
\begin{equation}\label{norm}
1\to \Z^n\to N_{\M(S)}(f){\xrightarrow{\rho_{\sigma}}} Q\to 1,
\end{equation}
such that $Q\subseteq \M(\sqcup_{i=1}^a \widehat S_i)\times A$,  where $A\subseteq N_{\M(\sqcup_{j=a+1}^k \widehat S_k)}((\hat f_{a+1},  \dots , \hat f_{k}))$ is a finite extension of $\prod_{j=a+1}^k V_j$.
\end{proposition}
\begin{proof}  The extension (\ref{restrict_cent}) has already been established in   \cite[Proposition 4.12]{JT}.  The kernel is generated by the Dehn twists about the curves $\alpha_1, \dots, \alpha_n$ which are fixed by $C_{\M(S)}(f)^0$, and hence is central.

By Lemma 3.8 of \cite{JT}, for every $g\in N_{\M(S)}(f)$, we have $g(\sigma)=\sigma$. Therefore, $N_{\M(S)}(f)\subseteq \M(S)_{\sigma}$. Let $Q=\rho_{\sigma}(N_{\M(S)}(f))$. Note that $Q$ is a finite extension of $\prod_{i=1}^a \M(\widehat S_i, \Omega_i)\prod_{j=a+1}^k V_j$ and it is contained in $N_{\M(\widehat S_\sigma)}(\rho_{\sigma}(f))$. So, to obtain (\ref{norm}), it suffices to show that  any $g\in \M(\widehat S_{\sigma})$ that normalises $\rho_{\sigma}(f)$ is contained in the subgroup $\M(\sqcup_{i=1}^a \widehat S_i)\times \M(\sqcup_{j=a+1}^k \widehat S_k)$ of $\M(\widehat S_{\sigma})$.

Suppose this is not the case, and  say $g$ maps $\widehat S_1$ diffeomorphically onto $\widehat S_k$. Then for any $x\in \widehat S_1$, we have $\rho_{\sigma}(f)g(x)=g\rho_{\sigma}(f)^{\pm 1}(x)=g(x)$. This shows that $\widehat f_k$ is the identity on $\widehat S_k$ which is a contradiction.
\end{proof}

\begin{lemma}\label{vcd} Let $1\to \Z^n\to G\to Q\to 1$ be an extension of groups where $Q$ is finitely generated with $\vcd Q=k < \infty$. Then $\vcd G= n+ k$.
\end{lemma}
\begin{proof} We can assume $Q$ is torsion-free. Now we apply a result of Fel'dman, see \cite[Theorem 5.5]{bieri}
\end{proof}

We will quite often make use of the following corollary. 

\begin{corollary}\label{nice} Let $1\to \Z^n\to G\to Q\to 1$ be an extension of groups where $Q$  has a cocompact model for $\ue Q$ of dimension $\vcd Q=k<\infty$. Then $G$ has a cocompact model for $\ue G$ of dimension $\ugd G=\vcd G =n+k$.  
\end{corollary}
\begin{proof} Let $\FF$ be the family of finite subgroup of $G$ and  $\GG$ be the family of the preimages of all the finite subgroups of $Q$ under the projection of $G$ onto $Q$. Applying Proposition \ref{trans} together with the general fact that finitely generated virtually free abelian groups of rank $m$ have the Euclidean space of dimension $m$ as a cocompact model for proper actions, we obtain that $G$ has a cocompact model for $\ue G$ of dimension $n+k$. Lemma \ref{vcd} finishes the claim.
\end{proof}

The next proposition concerns surfaces with boundary. Suppose $S$ has $b\ne 0$  boundary  components $\beta_1, \dots, \beta_b$. Note that the capping homomorphism $\theta_S:\M(S)\to \M(\widehat{S}, \Omega)\subseteq \M(\widehat{S})$ where $\Omega$ is the set of punctures of $\widehat{S}$ that come from capping the boundary components.

\begin{proposition}\label{boundary} Let $S$ be a compact orientable surface with nonempty boundary, finitely many punctures and $\chi(S)< 0$. Suppose $f\in \M(S)$ generates an infinite cyclic subgroup $H$. Then there exists $l\in \N$ such that $f^l\in \M(S)[m]$, $m\geq 3$, with  $N_{\M(S)}[H]=N_{\M(S)}(f^l)$ and 
\begin{equation}\label{boundary_seq}1\to \Z^b\to N_{\M(S)}(f^l)\xrightarrow{\theta_S} N_{\M(\widehat{S}, \Omega)}(\theta_S(f^l))\to 1,\end{equation}
where $\Z^b=\langle T_{\beta_1}, \dots, T_{\beta_b}\rangle$ and $\theta_S(f^l)$ is either trivial or of infinite order.
\end{proposition}
\begin{proof} By replacing $f$ with a sufficiently large power, we can assume that $l=1$ and $f\in \M(S)[m]$, $m\geq 3$ with $\theta_S(f)$ either trivial or of infinite order. Restricting the capping homomorphism to $N_{\M(S)}[H]$ we have
$$1\to \Z^b\to N_{\M(S)}[H]\xrightarrow{\theta_S} Q\to 1,$$ where $Q\subseteq \M(\widehat{S}, \Omega)$.  Note that if  $\theta_S(f)$ is trivial, then $N_{\M(S)}[H]= N_{\M(S)}(H)=\M(S)$ and $Q=\M(\widehat{S}, \Omega)$ as desired. Otherwise, $\langle\theta_S(f)\rangle$ is infinite cyclic and $Q \subseteq N_{\M(\widehat{S}, \Omega)}[\langle \theta_S(f)\rangle ]= N_{\M(\widehat{S}, \Omega)}(\theta_S(f))$ by Proposition \ref{pin} applied to $\widehat{S}$ (replacing $f$ with a sufficiently large power if necessary). Since $N_{\M(\widehat{S}, \Omega)}(\theta_S(f))\subseteq Q$ (see for example \cite[eq.~(25)]{JT}), we deduce $Q= N_{\M(\widehat{S}, \Omega)}(\theta_S(f))$. Hence, $N_{\M(S)}[H]=N_{\M(S)}(f)$.
\end{proof}

\begin{proposition}\label{cmm} Let  $G=\M(S)$ where $S$ is a compact orientable surface $S$ with $\chi(S)< 0$ possibly with finitely many punctures and boundary components. Denote by $\FF$ the family of finite subgroups of $G$. Let $H$ be an infinite cyclic subgroup of $G$. Then the commensurator $N_G[H]$  has  cocompact models for $E_{\FF\cap N_G[H]}N_G[H]$ and $E_{\FF[H]}N_G[H]$ of dimensions $\vcd N_G[H]$ and $\vcd N_G[H]-1$, respectively. 
\end{proposition}
\begin{proof} By Propositions \ref{pin} and \ref{boundary}, we can assume $N_G[H] = N_G(H)$ where $H=\langle f\rangle$ and $f\in \M(S)[m]$, $m\geq 3$.

{\it Suppose first that the boundary of $S$ is empty.} By the Nielsen-Thurston Classification Theorem, $f$ is either pseudo-Anosov or reducible. 

If $f$ is pseudo-Anosov, then by \cite[Theorem 1]{mcCarthy}, $N_G(f)$ is virtually cyclic. So, a Euclidean line and a point are (cocompact) models for $E_{\FF\cap N_G(f)}N_G(f)$ and $E_{\FF[H]}N_G(f)$, respectively. 

If $f$ is reducible, say with the canonical reduction system $\sigma$, then $N_G(f)$ satisfies (\ref{norm}) of Proposition \ref{sequence}:
\begin{equation*}
1\to \Z^n\to N_G(f){\xrightarrow{\rho_{\sigma}}}Q \to 1,
\end{equation*}
such that $Q$ is a finite index subgroup of 
 $$P:=\M(\sqcup_{i=1}^a \widehat S_i)\times A$$ 
 where $A\subseteq N_{\M(\sqcup_{j=a+1}^k \widehat S_k)}((\hat f_{a+1},  \dots , \hat f_{k}))$ is virtually free abelian.  Let $ L:=\M(\sqcup_{i=1}^a \widehat S_i)$. To show that  $N_G(f)$ has  a cocompact model for $E_{\FF\cap N_G(f)}N_G(f)$ of dimension $\vcd N_G(f)$, by Corollary \ref{nice},  it suffices to show that $Q$ has a cocompact model for $\ue Q$ of dimension $\vcd Q$. 

Since $Q$ is a finite index subgroup in $P$, it is enough to show that $P$ has  cocompact models for  $\ue P$  of dimension
$\vcd P=\vcd Q$. Since $A$ is virtually free abelian, applying Lemma \ref{vcd}, we obtain that $\vcd P=\vcd L + \vcd A$. By Proposition \ref{non-conn}, there is a cocompact model $X_L$ for $\ue L$ of dimension $\vcd L$. Since $A$ is finitely generated virtually free abelian, there is a cocompact model $X_A$ for $\ue A$ of dimension $\vcd A$. Then $X_L\times X_A$ is a cocompact model for the classifying space $\ue P$ of dimension $\vcd P$ as desired.

To establish the second claim, observe that  under the natural projection of $N_G(H)$ onto $W_G(H)=N_G(H)/H$, a classifying space $\ue W_G(H)$ becomes a model for a classifying space of $N_G(H)$ for $E_{\FF[H]}N_G(H)$. Hence its suffices to show that there is a cocompact model for $\ue W_G(H)$ of dimension at most $\vcd N_G(H)-1$.

By Proposition 5.5, we only need to consider two case: $\rho_{\sigma}(f)$ is the identity or it has infinite order. First, suppose $\rho_{\sigma}(f)$ is the identity, that is $H\leq \Z^n$. By Corollary 4.9 of \cite{JT}, we can assume that $\Z^n/H\cong \Z^{n-1}$. Then
\begin{equation}\label{weyl}
1\to \Z^{n-1}\to W_G(f){\xrightarrow{\overline{\rho}_{\sigma}}}Q \to 1,
\end{equation}

Again, using Corollary \ref{nice}, there is a cocompact model of dimension $n-1+\vcd Q=\vcd N_G(f)-1$.

Now, suppose $\rho_{\sigma}(f)$ has infinite order.  Then
\begin{equation*}
1\to \Z^{n}\to W_G(f){\xrightarrow{\overline{\rho}_{\sigma}}}Q/{\Z}\to 1,
\end{equation*}
where $\overline{\rho}_{\sigma}(f)$ generates $\Z\leq Q$. By Corollary \ref{nice}, it suffices to show then that $Q/{\Z}$ has a cocompact model for $\ue (Q/{\Z})$ of dimension at most $\vcd Q$. Note that, from (\ref{norm}), it follows that $\Z$ is normal in $P$. Therefore, it remains to show that $P/{\Z}$ has a  cocompact model for  $\ue (P/{\Z})$  of dimension $\vcd P$. But $P/{\Z}\cong L\times (A/{\Z})$ and  a similar argument as above gives a cocompact model $X_L\times X_{A/{\Z}}$ for $\ue (P/{\Z})$ of dimension $\vcd P-1$.

{\it Suppose $S$ has nonempty boundary.}  The proof easily reduces to the case of empty boundary. By (\ref{boundary_seq}) of Proposition \ref{boundary}, we have the central extension
\begin{equation}\label{here}1\to \Z^b\to N_{\M(S)}(H)\xrightarrow{\theta_S} N_{\M(\widehat{S}, \Omega)}(\theta_S(f))\to 1.\end{equation}
When $\theta_S(f)$ has infinite order, then the quotient of (\ref{here})  by $H$, gives 
\begin{equation}\label{there}1\to \Z^b\to W_{\M(S)}(H)\rightarrow W_{\M(\widehat{S}, \Omega)}(\theta_S(f))\to 1.\end{equation}
In case $\theta_S(f)$ is trivial, note that $N_{\M(S)}(H)=\M(S)$. Replacing $H=\langle f \rangle$ with a commensurable subgroup if necessary, we can assume that $f\in \Z^b$ such that $\Z^b/H\cong \Z^{b-1}$. Thus, we obtain
\begin{equation}\label{really}1\to \Z^b\to N_{\M(S)}(H)\xrightarrow{\theta_S} \M(\widehat{S}, \Omega)\to 1,\end{equation}
\begin{equation}\label{almost}1\to \Z^{b-1}\to  W_{\M(S)}(H)\to \M(\widehat{S}, \Omega)\to 1.\end{equation} 

From the empty boundary case of the proposition applied to $\widehat S$ and  Corollary \ref{nice} applied to (\ref{here}), (\ref{there}), (\ref{really}), and (\ref{almost}), we obtain the desired result.
\end{proof}

We need the following complete computation of the virtual cohomological dimension of $\M(S)$ for surfaces with negative Euler characteristic  by Harer.

\begin{theorem}[{\cite[Thm.~4.1]{harer}}]\label{Harer}  Let $S=S^r_{g, b}$ be an oriented surface of genus $g$, $b$ boundary components and $r$ punctures and recall $\chi(S)=2-2g - b - r$. If $\chi(S)<0$, then
 $$\vcd \M(S)=\left\{ \begin{array}{ll}
4g+2b+r-4 & g>0, r+b>0\\
4g -5 & r +b =0\\
2b+r - 3 & g=0.\\
\end{array}\right.$$       
\end{theorem}

We are now ready to prove our main theorem.

\begin{proof}[Proof of Theorem \ref{main}] Set $G=\M(S)$ and let $\FF$ and $\Vc$ be the families of finite and virtually cyclic subgroups of $G$, respectively, equipped with the commensurabilty relation on $\Vc \smallsetminus \FF$. By \cite[Cor.~1.3]{AM},  there is a cocompact model for $\underline{E}G$ of minimal dimension $\ugd G =\vcd G$.
By Proposition \ref{cmm}, for each infinite cyclic subgroup $H \leq G$,  there are cocompact models for both $E_{\FF\cap N_G[H]}N_G[H]$ and $E_{\FF[H]}N_G[H]$ of dimension $\vcd G$.

Applying Theorem \ref{comm-thm}, we obtain a hierarchically cocompact model for $\uueg$ of dimension $\vcd G+ 1$.

To prove the second part of the theorem, assume we have a closed surface $S$. If $g=1$, then $\M(S)\cong\mbox{SL}(2, \Z)$. So, by  \cite[Lemma 5.2]{DP}, $\uucd(\M(S))\geq 2$. Combining this with \cite[Proposition 9]{jpl} shows that  $\uugd \M(S)=\uucd \M(S)=2.$

Now, suppose that $S$ is closed and $g\geq 2$. Then $\vcd \M(S)=4g-5$. Let $\sigma=\{[\alpha], [\beta]\}$ where $\alpha$ and $\beta$ are essential curves that separate $S$ into a pair of pants and a surface of genus $g-1$ (see Figure \ref{fig}). We will show that $\uucd \M(S)_\sigma^0=4g-4$. This will imply that $\uucd \M(S)\geq 4g-4$ and by the first part of the theorem, we will obtain that $\uugd \M(S)=\uucd \M(S)=4g-4$.


Using the cutting homomorphism, there is a short exact sequence 
$$ 1\to \Z^2\to \M(S)_\sigma^0{\xrightarrow{\rho_{\sigma}}} \M(S^3_{0, 0})\times \M(S^1_{g-1, 0})\to 1, $$ where $\ker (\rho_{\sigma})=\langle T_{\alpha}, T_{\beta}\rangle$.  By Lemma \ref{vcd}, 
\begin{align*}
\vcd \M(S_\sigma)^0&=2+\vcd \M(S^0_{3, 0})+ \vcd  \M(S^1_{g-1, 0})\\
&=2+0+4(g-1)+1-4\\
&=4g-5.
\end{align*}
By Proposition \ref{non-conn}, it follows that there is a cocompact model of for ${\underline E} \M(S_\sigma)^0$ of dimension $\vcd \M(S_\sigma)^0=4g-5$ 
which is therefore the same as $\ugd \M(S_\sigma)^0$.

Set $C=\M(S_\sigma)^0\cap \M(S)[3],$ where $\M(S)[3]$ is the level $3$ congruence subgroup of $\M(S)$ which is torsion-free \cite[Theorem 6.9]{FM}. Hence $C$ is a finite index torsion-free subgroup of $\M(S_\sigma)^0$. Let $\FF$ and $\Vc$ be the families of finite (in this case trivial) and virtually cyclic subgroups of $C$, respectively, equipped with the commensurabilty relation on $\Vc \smallsetminus \FF$. Define $$M:\OVC C \to \Z\mbox{ - mod}: C/H \to (\Z C)^H.$$ The long exact cohomology sequence associated to the push-out of Proposition \ref{lw-main} gives us
$$\prod_{[H]\in I} H^{(4g-5)}_{\FF[H]}(N_C[H], M)\oplus H^{(4g-5)}(C, \Z C)\xrightarrow{i^*} \prod_{[H]\in I} H^{(4g-5)}(N_C[H], \Z C)$$$$\to H^{(4g-4)}_{\Vc}(C, M).$$

By Proposition \ref{cmm}, this reduces to 
$$H^{(4g-5)}(C, \Z C)\xrightarrow{i^*}  \prod_{[H]\in I} H^{(4g-5)}(N_C[H], \Z C)\to H^{(4g-4)}_{\Vc}(C, M).$$
Since $C$ is of type $F$, note that $H^{(4g-5)}(C, \Z C)\ne 0$.   Consider the infinite cyclic subgroups $H_1=\langle T_{\alpha}T_{\beta}\rangle$ and $H_2=\langle T^2_{\alpha}T_{\beta}\rangle$. Since every element of $C$ fixes the curves $\alpha$ and $\beta$, it must commute with both $H_1$ and $H_2$.  Thus, $N_C[H_1]=N_C[H_2]=C$ and $H_1$, $H_2$ represent distinct classes in $I$.  The composition of $i^*$ with the projection of $\prod_{[H]\in I} H^{(4g-5)}(N_C[H], \Z C)$ onto the two factors corresponding to these subgroups is the diagonal map $$\Delta: H^{(4g-5)}(C, \Z C)\to H^{(4g-5)}(C, \Z C)\oplus H^{(4g-5)}(C, \Z C)$$ which is not surjective. Therefore, $i^*$ cannot be surjective and we obtain that $H^{(4g-4)}_{\Vc}(C, M)\ne 0$.
\end{proof}

\vspace{-1cm}
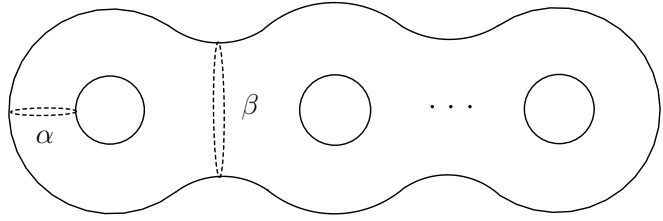
\begin{figure}[!h]

  \begin{minipage}[t]{0.75\linewidth}
    \hspace*{10\linewidth}
    \rule[0cm]{0pt}{0cm}
\scalebox{0.45}{\begin{tikzpicture}[line cap=round,line join=round,>=triangle 45,x=1cm,y=1cm]\clip(-12.096363636363629,-8.842683982683981) rectangle (14.778095238095258,6.430043290043288);\draw [shift={(-8.14,-1.18)},line width=1.2pt]  plot[domain=0.9883083732263253:5.374815274217841,variable=\t]({1*3.017614952242914*cos(\t r)+0*3.017614952242914*sin(\t r)},{0*3.017614952242914*cos(\t r)+1*3.017614952242914*sin(\t r)});\draw [shift={(4.86,-1.14)},line width=1.2pt]  plot[domain=-2.1532842803634678:2.1771792131708594,variable=\t]({1*3.017614952242914*cos(\t r)+0*3.017614952242914*sin(\t r)},{0*3.017614952242914*cos(\t r)+1*3.017614952242914*sin(\t r)});\draw [shift={(-1.64,-1.1)},line width=1.2pt]  plot[domain=0.8850668158886104:2.316215803069055,variable=\t]({1*3.1267874887814173*cos(\t r)+0*3.1267874887814173*sin(\t r)},{0*3.1267874887814173*cos(\t r)+1*3.1267874887814173*sin(\t r)});\draw [shift={(-1.64,-1.1)},line width=1.2pt]  plot[domain=4.027567734442089:5.440073530173172,variable=\t]({1*3.0986448650982896*cos(\t r)+0*3.0986448650982896*sin(\t r)},{0*3.0986448650982896*cos(\t r)+1*3.0986448650982896*sin(\t r)});\draw [shift={(-5.02,3.22)},line width=1.2pt]  plot[domain=4.0520827310763625:5.26926057997655,variable=\t]({1*2.380336110720501*cos(\t r)+0*2.380336110720501*sin(\t r)},{0*2.380336110720501*cos(\t r)+1*2.380336110720501*sin(\t r)});\draw [shift={(-4.9,-5.38)},line width=1.2pt]  plot[domain=0.9658050665645345:2.220810230823389,variable=\t]({1*2.28569464277274*cos(\t r)+0*2.28569464277274*sin(\t r)},{0*2.28569464277274*cos(\t r)+1*2.28569464277274*sin(\t r)});\draw [shift={(-3.8,-1.16)},line width=1.2pt]  plot[domain=4.797651849825218:4.797651849825218,variable=\t]({1*2.3485314560380046*cos(\t r)+0*2.3485314560380046*sin(\t r)},{0*2.3485314560380046*cos(\t r)+1*2.3485314560380046*sin(\t r)});\draw [shift={(1.72,3.64)},line width=1.2pt]  plot[domain=4.175781405742072:5.275775752963978,variable=\t]({1*2.6994073423623934*cos(\t r)+0*2.6994073423623934*sin(\t r)},{0*2.6994073423623934*cos(\t r)+1*2.6994073423623934*sin(\t r)});\draw [shift={(1.66,-5.3)},line width=1.2pt]  plot[domain=0.8168343447537226:2.1519776650533373,variable=\t]({1*2.2497110925627757*cos(\t r)+0*2.2497110925627757*sin(\t r)},{0*2.2497110925627757*cos(\t r)+1*2.2497110925627757*sin(\t r)});\draw [line width=1.2pt] (-8.2,-1.14) circle (1.0031948963187611cm);\draw [line width=1.2pt] (-1.62,-1.14) circle (1.0376897416858306cm);\draw [line width=1.2pt] (4.93,-1.11) circle (1.0198039027185568cm);\draw [rotate around={-89.43222823640274:(-5.019521039774485,-1.130129971810492)},line width=1.2pt,dash pattern=on 3pt off 3pt] (-5.019521039774485,-1.130129971810492) ellipse (1.990713927528157cm and 0.1566582325717562cm);\draw [rotate around={0.6543659282899769:(-10.141226855239825,-1.189277955970762)},line width=1.2pt,dash pattern=on 3pt off 3pt] (-10.141226855239825,-1.189277955970762) ellipse (0.9579748255313298cm and 0.11054077464068775cm);\draw (-4.504155844155833,-0.513679653679655) node[anchor=north west] {\scalebox{2}{${\beta}$}};\draw (-10.516709956709948,-1.5967965367965378) node[anchor=north west] {\scalebox{2}{$\alpha$}};\begin{scriptsize}\draw [fill=black] (1.22,-1.04) circle (1pt);\draw [fill=black] (1.76,-1.04) circle (1pt);\draw [fill=black] (2.32,-1.04) circle (1pt);\end{scriptsize}\end{tikzpicture}}
  \end{minipage}%
\vspace{-1cm}

\caption{The closed surface $S$ of genus $g\geq 2$ and the essential curves $\alpha$, $\beta$ separating it into a pair of pants and a surface of genus $g-1$.}
\label{fig}
\end{figure}

\end{document}